\documentclass[11pt,a4paper]{article}
\usepackage{mathrsfs}
\usepackage{mathrsfs}
\usepackage{mathrsfs}
\usepackage{amsfonts}
\usepackage{amsfonts}
\usepackage{amsfonts}
\usepackage{mathrsfs}
\usepackage{amssymb}
\usepackage{color}
\usepackage{amsmath}
\newtheorem{theorem}{Theorem}[section]

\numberwithin{equation}{section}

\newtheorem{definition}[theorem]{Definition}
\newtheorem{example}[theorem]{Example}

\newtheorem{lemma}[theorem]{Lemma}

\newenvironment{proof}[1][Proof]{\textbf{#1.}}{\ \rule{0.5em}{0.5em}}%

\begin{document}
\parindent 9mm
\title{Existence of mild solutions for Riemann-Liouville fractional differential
equations with nonlocal conditions
\thanks{This work was
supported by the Natural Science Foundation of China (Grant No.
11301412 and 11131006), and the Fundamental Research Funds for the Central Universities (Grant No. 2012jdhz52)}
}
\author{Zhan-Dong Mei \thanks{Corresponding
author, School of Mathematics and Statistics, Xi'an Jiaotong
University, Xi'an 710049, China; Email: zhdmei@mail.xjtu.edu.cn}  \ \ \
Ji-Gen Peng \thanks{School of Mathematics and Statistics, Xi'an Jiaotong
University, Xi'an 710049, China; Email: jgpeng@mail.xjtu.edu.cn} }


\date{}
\maketitle \thispagestyle{empty}
\begin{abstract}
%
%
In this paper, we are concerned with the mild solutions of
Riemann-Liouville fractional differential equations with nonlocal
conditions in Banach space. We use Banach contraction principle to
prove the existence and uniqueness. Moreover, we derive the
existence by using Krasnoselkii's theorem.
An illustrative example is presented.
\vspace{0.5cm} 

%
%

\end{abstract}


\section{Introduction}

Consider inhomogeneous abstract Riemann-Liouville fractional differential equations
described by
\begin{align}\label{equ}
  \left\{
     \begin{array}{ll}
       D^\alpha_t u(t)=A u(t)+f(t,t^{1-\alpha}u(t),(Ku)(t)), & \hbox{ } t\in(0,T],\\
       \tilde{u}(0)=x+g(u), & \hbox{ }
     \end{array}
   \right.
\end{align}
where $0<\alpha\leq 1$, $u(\cdot)$ is the state, $A:D(A)\subset
X\rightarrow X$ is a closed and densely defined linear operator,
$(X,\|\cdot\|)$ is a Banach space, $D(A)$ is the domain of $A$
endowed with the graph norm $\|\cdot\|_{D(A)}=\|\cdot\|+\|A\cdot\|$,
$D_t^\alpha$ is the $\alpha$-order Riemann-Liouville fractional
derivative operator, $(Ku)(t)=\int_0^tr(t,s)u(s)ds$, $f:[0,T]\times
X\times X\rightarrow X$, $\tilde{u}(0)=\lim_{t\rightarrow 0^+}\Gamma(\alpha)t^{1-\alpha}u(t)$, $x\in X$, $g$ is a function from a certain function space on $X$ to $X$.

Fractional differential equations have received increasing attention
because the behavior of many physical systems, such as fluid flows,
electrical networks, viscoelasticity, chemical physics,
electron-analytical chemistry, biology, control theory, can be
properly described by using the fractional order system theory etc.
(see \cite{Eidelman2004,Lakshmikantham2009,
Meerschaert2009,Metzler2001,Podlubny1999}). Fractional derivatives
appear in the theory of fractional differential equations; they
describe the property of memory and heredity of materials, and it is
the major advantage of fractional derivatives compared with integer
order derivatives. Many of the references on fractional differential
equations were focused on the existence and/or uniqueness of
solutions for fractional differential equations
\cite{Bazhlekova2001,Bazhlekova2012,Eidelman2004,Fan2012,Ibrahim2011}.

The nonlocal Cauchy problem, an initial problem for the
corresponding equations with nonlocal initial data, was first
studied by Byszewski \cite{Byszewski1991}. Such problem has better
effects than the normal Cauchy problem with the classical initial
data because nonlocal condition can be applied in physics with
better effect in applications than the classical initial condition
since nonlocal conditions are usually more precise for physical
measurements than the classical initial condition (cf., e.g.,
\cite{Anguraj2009,Byszewski1991,Fu2004,Gatsori2004,Liang2009,Lin1996,Nguerekata2006,
Xiao2003} and references therein). Very recently, the existence and
uniqueness of solutions of Caputo fractional abstract differential
equations with a nonlocal initial condition were discussed by some
references (cf., e.g., Anguraj e.t. \cite{Anguraj2009}, Balachandran
e.t. \cite{Balachandran2009}, Li, e.t. \cite{Li2012}, Zhou, e.t.
\cite{Zhou2010}). N'Guerekata \cite{Nguerekata2009} studied the mild
solutions of fractional differential equations with nonlocal
conditions related to Riemann-Liouville derivative, which results in singularity at zero.
However, Li, Peng and Gao \cite{Li2012aa} pointed out that the
definition of the mild solution in \cite{Nguerekata2009} is
incorrect and the similar situation can be found in \cite{Hu2009}.
Motivated by this, in this paper, we will use
fractional resolvent developed by Li and Peng \cite{Li2012a} and
introduce a new norm to study the existence and uniqueness of
equation (\ref{equ}).

The arrangement of this paper is as follows. Sec. 2 is to introduce
some related preliminaries. In Se. 3, Banach
contraction principle is used to prove the existence and uniqueness and Krasnoselkii's theorem is used to derive the existence of the mild solutions of (\ref{equ}).
\section{Preliminaries}

Let $(X,\|\cdot\|)$ be a Banach space. For $q\geq1$, $L^q((0,T); X)$
denotes the space of all $X$-valued functions $u:(0,T)\rightarrow X$
with the norm $\|u\|_{L^{q}((0,T);X)}=
\bigg(\int_{0}^T\|u(t)\|^qdt\bigg)^{\frac{1}{q}}$. Denote by
$C_{1-\alpha}([0,T],X)$ all the functions such that $t\mapsto
t^{1-\alpha}u(t)$ is continuous on $[0,T]$ with the norm
$\|u\|_{C_{1-\alpha}([0,T];X)}=\sup_{t\in[0,T]}\|t^{1-\alpha}u(t)\|$.
Obviously, $L^{q}((0,T);X)$ and $C_{1-\alpha}([0,T],X)$ are Banach
spaces. Let $n\in N,\ 1\leq q<\infty$. Let $I=(0,T)$, or $I=[0,T]$,
or $I=(0,\infty)$. The Sobolev spaces $W^{n,p}(I; X)$ is defined as
follows (\cite[Appendix]{Brezis1973}):
\begin{equation*}
W^{n,p}(I; X)=\{u|\ \exists \varphi\in L^{p}(I;
X):u(t)=\sum_{k=0}^{n-1}c_{k}\frac{t^{k}}{k!}+\frac{t^{n-1}}{(n-1)!}\ast
\varphi(t),\ t\in I\}.
\end{equation*}
In this case, we have $\varphi(t)=u^{(n)}(t),\ c_{k}=u^{(k)}(0)$.\\\\

For the convenience of the readers, we shall introduce some
definitions and some fundamental properties of fractional calculus
theory, which can be fund in
\cite{Hilfer2000,Lakshmikantham2009,Podlubny1999, Srivastava2009}.

\begin{definition}
 For any $u\in L^{1}((0,T);X)$, the $\alpha$-order Riemann-Liouville
fractional integral of $u$ is defined by
\begin{equation}\label{0}
J_{t}^{\alpha}u(t)=\frac{1}{\Gamma(\alpha)}\int_{0}^{t}(t-\tau)^{\alpha-1}u(\tau)d\tau.
\end{equation}
\end{definition}
We denote $J_{t}^{0}u(t)=u(t)$. Obviously, the fractional integral
operators $\{J^\alpha_t\}_{\alpha\geq0}$ satisfies the semigroup
property $ J^\alpha_t J^\beta_t=J^{\alpha+\beta}_t,\  \alpha,
\beta\geq0. $

\begin{definition}
Let $\alpha\in(0,1)$. The the $\alpha$-order Riemann-Liouville
fractional derivative of $u$ is defined by
\begin{align*}
    D_t^\alpha
    u(t)=\frac{1}{\Gamma(1-\alpha)}\frac{d}{dt}\int_0^t(t-\sigma)^{-\alpha}u(\sigma)d\sigma.
\end{align*}
\end{definition}

\begin{definition}\cite{Li2012a}\label{de}
 Let $0<\alpha< 1$. A family $\{T(t)\}_{t>0}$ of bounded linear operators on Banach space $X$ is called an
 $\alpha$-order fractional resolvent
 if it satisfies the following assumptions:

(P1) for any $x\in X$, $\ T(\cdot)x\in C((0,\infty),X)$, and
\begin{equation}\label{clear}
\lim_{t\rightarrow 0+}\Gamma(\alpha)t^{1-\alpha}T(t)x=x \ \
\mbox{for all } \ x\in X;
\end{equation}

(P2) $T(s)T(t)=T(t)T(s) \ \ \mbox{for all } t,s\geq0;$

(P3) for all $t,s>0$, there holds
\begin{equation}\label{fi}
\ T(t) J_{s}^{\alpha}T(s)-
J_{t}^{\alpha}T(t)T(s)=\frac{t^{\alpha-1}}{\Gamma(\alpha)}J_{s}^{\alpha}T(s)
-\frac{s^{\alpha-1}}{\Gamma(\alpha)}J_{t}^{\alpha}T(t),
\end{equation}
where $J_t^\alpha$ is $\alpha$-order Riemann-Liouville fractional
integral operator.
\end{definition}

The generator $A$ of fractional resolvent $\{T(t)\}_{t> 0}$ is
defined by $$D(A)=\{x\in X: \mbox{the}\ \mbox{limt} \
\lim_{t\rightarrow
0^+}\frac{t^{1-\alpha}T(t)x-\frac{x}{\Gamma(\alpha)}}{t^\alpha}\
\mbox{exists}\},$$ and
$$Ax=
\lim_{t\rightarrow
0^+}\frac{t^{1-\alpha}T(t)x-\frac{x}{\Gamma(\alpha)}}{t^\alpha}.
$$

\begin{theorem}\cite{Li2012a}\label{Lith}
Let $\{T(t)\}_{t> 0}$ be a fractional resolvent and $A$ its
generator. Then, we have\\

 (a) For any $x\in X$, $\int_0^t
\frac{(t-s)^{\alpha-1}}{\Gamma(\alpha)}T(s)xds\in D(A)$ and
\begin{equation}\label{c1}
T(t)x=\frac{t^{\alpha-1}}{\Gamma(\alpha)}x+A\int_0^t
\frac{(t-s)^{\alpha-1}}{\Gamma(\alpha)}T(s)xds, \ t>0;
\end{equation}

(b) For any $x\in D(A)$,
\begin{equation}\label{c2}
T(t)x=\frac{t^{\alpha-1}}{\Gamma(\alpha)}x+\int_0^t
\frac{(t-s)^{\alpha-1}}{\Gamma(\alpha)}T(s)Axds, \ t>0;
\end{equation}

\end{theorem}

\begin{theorem}\cite{Krasnoselskii1964}\label{the}(Krasnoselskii)
 Let $B$ be a closed convex and nonempty subset of
a Banach space $X$. Let $\mathcal {A}$ and $\mathcal {B}$ be two
operators such that \\
(i) $\mathcal {A}u +\mathcal {B}v\in B$ whenever $u,v \in B$; \\
(ii)$\mathcal {A}$ is a contraction mapping;\\
(iii)$\mathcal {B}$ is compact and continuous. \\
Then there exists $z\in B$ such that $z = \mathcal {A}z + \mathcal
{B}z$.
\end{theorem}

\section{Existence of mild solution}

In this section, we shall prove the existence and uniqueness of
the mild solution of (\ref{equ}). To begin with, we introduce the notion mild solution.

\begin{definition}
A function $u\in C((0,T],X)$ is called a mild solution of equation
(\ref{equ}), if $J_t^\alpha u(t)\in D(A), \ t\in(0,T]$, and there
holds
\begin{align}\label{mild}
    u(t)=\frac{t^{\alpha-1}}{\Gamma(\alpha)}[x-g(u)]+
    AJ_t^\alpha u(t)+J_t^\alpha f(t,t^{1-\alpha}u(t),(Ku)(t)), \ t\in(0,T].
\end{align}
\end{definition}

\begin{lemma}
Suppose that $A$ generates a fractional resolvent $\{S(t)\}_{t>0}$.
Let $f\in C([0,T]\times X\times X)$. If $u\in C((0,T],X)$
is a mild solution of system (\ref{equ}), then
\begin{align}\label{mild2}
 u(t)=S(t)[x-g(u)]+\int_0^tS(t-s)f\big(s,s^{1-\alpha}u(s),(Ku)(s)d\sigma\big)ds;
\end{align}
conversely, if $u\in C((0,T],X)$ satisfies (\ref{mild2}), then $u$ is
a mild solution of system (\ref{equ}).
\end{lemma}

\begin{proof}\ \
Assume that $u\in C((0,T],X)$ is a mild solution of (\ref{equ}).
Then, $g_\alpha(t)*u(t)=J_t^\alpha u(t)\in D(A)$. By (\ref{c1}), it
follows that
\begin{align}\label{1}
\nonumber\frac{t^{\alpha-1}}{\Gamma(\alpha)}*u(t)=&\bigg(S(t)-
Ag_\alpha(t)*S(t)\bigg)*u(t)\\
\nonumber=& S(t)*u(t)-AS(t)*g_\alpha(t)*u(t)\\
\nonumber=& S(t)*u(t)-S(t)*A\bigg(g_\alpha(t)*u(t)\bigg)\\
\nonumber=& S(t)*(u(t)-AJ_t^\alpha u(t))\\
\nonumber=& S(t)*\bigg(g_\alpha(t)x-g_\alpha(t)*f(t)\bigg)\\
=&
\frac{t^{\alpha-1}}{\Gamma(\alpha)}*
\bigg(S(t)[x-g(u)]+\int_0^tS(t-s)f\big(s,s^{1-\alpha}u(s),(Ku)(s)d\sigma\big)ds\bigg).
\end{align}
By Titchmarsh's theorem, we have
\begin{align*}
 u(t)=S(t)[x-g(u)]+\int_0^tS(t-s)f\big(s,s^{1-\alpha}u(s),(Ku)(s)d\sigma\big)ds.
\end{align*}

Assume that $u\in C((0,T],X)$ satisfies (\ref{mild2}). Then, by (a) of Theorem \ref{Lith}, it follows that
\begin{align*}
 J_t^\alpha u(t)
 =J_t^\alpha S(t)[x-g(u)]+g_\alpha(t)*S(t)*f\big(t,t^{1-\alpha}u(t),(Ku)(t)\big)\in
 D(A)
\end{align*}
and
\begin{align*}
 AJ_t^\alpha u(t)
 =&\bigg(S(t)-\frac{t^{\alpha-1}}{\Gamma(\alpha)}\bigg)[x-g(u)]
 +\bigg(S(t)-\frac{t^{\alpha-1}}{\Gamma(\alpha)}\bigg)*f\big(t,t^{1-\alpha}u(t),(Ku)(t)\big)\\
 =& S(t)(x-g(u))-\frac{t^{\alpha-1}}{\Gamma(\alpha)}(x-g(u))
 +S(t)*f\big(t,t^{1-\alpha}u(t),(Ku)(t)\big)\\
 &-\frac{t^{\alpha-1}}{\Gamma(\alpha)}*f\big(t,t^{1-\alpha}u(t),(Ku)(t)\big)\\
 =&
 u(t)-\frac{t^{\alpha-1}}{\Gamma(\alpha)}(x-g(u))
 -\frac{t^{\alpha-1}}{\Gamma(\alpha)}*f\big(t,t^{1-\alpha}u(t),(Ku)(t)\big).
\end{align*}
The proof is therefore completed.
\end{proof}

\begin{lemma}\label{iqq}
Assume that $x,\ y>0$ and $0<\gamma<1$. Then
\begin{align}\label{aa}
    |x^\gamma-y^\gamma|\leq |x-y|^\gamma.
\end{align}

\end{lemma}

\begin{proof}\ \
Assume that $x>y$. Then (\ref{aa}) is equivalent to
\begin{align}\label{aaa}
   \bigg( \frac{x}{y}\bigg)^\gamma-1\leq \bigg(\frac{x}{y}-1\bigg)^\gamma.
\end{align}
Define function $g(z):=(z+1)^{\gamma}-z^\gamma-1, \ z>0$. We can
easily obtain that the derivative of $g$ at each $z>0$ satisfies
that $g'(z)=\gamma (z+1)^{\gamma-1}-\gamma z^{\gamma-1}<0.$ This
means that $g(z)$ is monotone-decreasing function. So we have
\begin{align*}
    g\bigg(\frac{x}{y}-1\bigg)=\bigg(\frac{x}{y}\bigg)^\gamma
    -\bigg(\frac{x}{y}-1\bigg)^\gamma-1\leq g(0)=0,
\end{align*}
that is, (\ref{aaa}) holds. The proof is therefore completed.
\end{proof}

Denote $M=\max_{t\in[0,T]}\|t^{1-\alpha}S(t)\|$, $N=\max_{0\leq
s\leq t\leq T}r(t,s)$, $P=\max_{0\leq s\leq t\leq T}m(t,s)$.\\
In order to derive our main results, the following hypothesis are introduced:\\
$(H_1)$ There exists two constants $\alpha_1,\
\alpha_2\in(0,\alpha)$ and real-valued functions $m_1(t)\in
L^{\frac{1}{\alpha_1}}([0,T],R)$, $ m_2(t)\in
L^{\frac{1}{\alpha_2}}([0,T],R) $ such that
\begin{align*}
    \|f(t,u_1,u_2)-f(t,v_1,v_2)\|\leq m_1(t)\|u_1-v_1\|+m_2(t)\|u_2-v_2\|.
\end{align*}\\
 ($H_2$) There exists a constant $\alpha_3\in(0,\alpha)$ and
real-valued function $h(t)\in L^{\frac{1}{\alpha_3}}([0,T],R)$ such
that
\begin{align*}
    \|f(t,u_1,v_1)\|\leq h(t), \ t\in [0,T],\ u_1,v_1\in X.
\end{align*}\\
$(H_3)$ There exists a constant $b$ such that
\begin{align*}
    \|g(u)-g(v)\|\leq b\|u-v\|_*, \ u,v\in C_{1-\alpha}([0,T],X).
\end{align*}\\

We denote $M_1=\|m_1\|_{L^{\frac{1}{\alpha_1}}([0,T],R)}$,
$M_2=\|m_2\|_{L^{\frac{1}{\alpha_2}}([0,T],R)}$ and
$H=\|h\|_{L^{\frac{1}{\alpha_3}}([0,T],R)}$.

\begin{theorem}\label{3.4}
Assume that $(H_1)-(H_3)$ hold. If
\begin{align*}
    \Omega:=Mb+
    \frac{MM_1T^{1-\alpha_1}}{\big(\frac{\alpha-\alpha_1}{1-\alpha_1}\big)^{1-\alpha_1}}
     +\frac{MM_2NT^{\alpha+1-\alpha_2}}{\alpha\big(\frac{\alpha-\alpha_2}{1-\alpha_2}\big)^{1-\alpha_2}}
    <1,
\end{align*}
then system (\ref{equ}) has a unique solution.
\end{theorem}

\begin{proof}\ \
Consider the following operator:
\begin{align} \label{mainoperator}
   \nonumber(\mathcal {N} u)(t)=&S(t)[x-g(u)]\\
   &+\int_0^tS(t-s)f\big(s,s^{1-\alpha}u(s),(Ku)(s)\big)ds,\ u\in C_{1-\alpha}([0,T],X),\ t>0.
\end{align}
We shall first verify that $\mathcal
{N}:C_{1-\alpha}([0,T],X)\rightarrow C_{1-\alpha}([0,T],X)$. Let
$u\in C_{1-\alpha}([0,T],X)$, $t,\delta>0,$ $t+\delta\leq T$. We compute
\begin{align}\label{zero}
   \nonumber &\bigg\|t^{1-\alpha}(\mathcal {N} u)(t)-\frac{1}{\Gamma(\alpha)}[x-g(u)]\bigg\|\\
\nonumber\leq&
\bigg\|\bigg(t^{1-\alpha}S(t)-\frac{1}{\Gamma(\alpha)}\bigg)[x-g(u)]\bigg\|
+\bigg\|t^{1-\alpha}\int_0^tS(t-s)f\big(s,s^{1-\alpha}u(s),(Ku)(s)\big)ds\bigg\|\\
\nonumber\leq&
\bigg\|\bigg(t^{1-\alpha}S(t)-\frac{1}{\Gamma(\alpha)}\bigg)[x-g(u)]\bigg\|
+t^{1-\alpha}\int_0^t(t-s)^{\alpha-1}
\|(t-s)^{1-\alpha}S(t-s)\|\|h(s)\|ds\\
\nonumber\leq&
\bigg\|\bigg(t^{1-\alpha}S(t)-\frac{1}{\Gamma(\alpha)}\bigg)[x-g(u)]\bigg\|
+Mt^{1-\alpha}
\bigg(\int_0^t(t-s)^{\frac{\alpha-1}{1-\alpha_3}}ds\bigg)^{1-\alpha_3}
\bigg(\int_0^t\|h(s)\|^{\frac{1}{\alpha_3}}ds\bigg)^{\alpha_3}\\
=&
\bigg\|\bigg(t^{1-\alpha}S(t)-\frac{1}{\Gamma(\alpha)}\bigg)[x-g(u)]\bigg\|
+\frac{Mt^{1-\alpha_3}}{\bigg(\frac{\alpha-\alpha_3}{1-\alpha_3}\bigg)^{1-\alpha_3}
} \|h\|_{L^{\frac{1}{\alpha_3}}}.
\end{align}
Since $\lim_{t\rightarrow 0^+}\Gamma(\alpha)t^{1-\alpha}S(t)[x-g(u)]=x-g(u)$, the
inequality (\ref{zero}) implies that the limit $\lim_{t\rightarrow
0^+}t^{1-\alpha}(\mathcal {N} u)(t)$ exists and
\begin{align}\label{11}
    \lim_{t\rightarrow
0^+}t^{1-\alpha}(\mathcal {N} u)(t)=\frac{1}{\Gamma(\alpha)}[x-g(u)].
\end{align}

Denote $p_s(t)=t^{1-\alpha}(t-s)^{\alpha-1}$ and
$q_s(t)=(t-s)^{1-\alpha}S(t-s)$. Using Lemma \ref{iqq}, we deduce
the following inequality,
\begin{align}\label{no}
    \nonumber &\big\|(t+\delta)^{1-\alpha}S(t+\delta-s)-t^{1-\alpha}S(t-s))\big\|\\
    \nonumber =&\big\|p_s(t+\delta)q_s(t+\delta)-p_s(t)q_s(t)\big\|\\
    \nonumber \leq & \big\|p_s(t+\delta)q_s(t+\delta)-p_s(t+\delta)q_s(t)\big\|+\big\|p_s(t+\delta)q_s(t)-p_s(t)q_s(t)\big\|\\
    \nonumber \leq&|p_s(t+\delta)|\big\|q_s(t+\delta)-q_s(t)\big\|
    +M|p_s(t+\delta)-p_s(t)|\\
    \nonumber \leq& |p_s(t+\delta)|\big\|q_s(t+\delta)-q_s(t)\big\|
    +M|(t+\delta)^{1-\alpha}(t+\delta-s)^{1-\alpha}-(t+\delta)^{1-\alpha}(t-s)^{1-\alpha}|\\
    \nonumber &+M|(t+\delta)^{1-\alpha}(t-s)^{1-\alpha}-t^{1-\alpha}(t-s)^{1-\alpha}|\\
    \nonumber \leq & |p_s(t+\delta)|\big\|q_s(t+\delta)-q_s(t)\big\|
    +M(t+\delta)^{1-\alpha}|(t+\delta-s)^{1-\alpha}-(t-s)^{1-\alpha}|\\
    \nonumber &+M(t-s)^{1-\alpha}|(t+\delta)^{1-\alpha}-t^{1-\alpha}|\\
    \nonumber \leq& \sup_{0\leq s\leq t\leq T-\delta}\big\|q_s(t+\delta)-q_s(t)\big\| (t+\delta)^{1-\alpha}(t+\delta-s)^{\alpha-1}
    +M(t+\delta)^{1-\alpha}\delta^{1-\alpha}
    +M(t-s)^{1-\alpha}\delta^{1-\alpha}\\
    \leq& \sup_{0\leq s\leq t\leq T-\delta}\big\|q_s(t+\delta)-q_s(t)\big\| T^{1-\alpha}(t+\delta-s)^{\alpha-1}
    +2MT^{1-\alpha}\delta^{1-\alpha}.
\end{align}
Combining inequality (\ref{no}) and Lemma \ref{iqq}, we have
\begin{align}\label{conti}
   \nonumber &\bigg\|(t+\delta)^{1-\alpha}\int_0^{t+\delta}S(t+\delta-s)f\big(s,s^{1-\alpha}u(s),(Ku)(s)\big)ds\\
    \nonumber &-t^{1-\alpha}\int_0^tS(t-s)f\big(s,s^{1-\alpha}u(s),(Ku)(s)\big)ds\bigg\|\\
   \nonumber \leq& \int_0^t\big\|(t+\delta)^{1-\alpha}S(t+\delta-s)-t^{1-\alpha}S(t-s))\big\|h(s)ds
   +(t+\delta)^{1-\alpha}\int_t^{t+\delta}\|S(t+\delta-s)\|h(s)ds\\
   \nonumber\leq &\sup_{0\leq s\leq t\leq T-\delta}\big\|q_s(t+\delta)-q_s(t)\big\|
   T^{1-\alpha}\int_0^t(t+\delta-s)^{\alpha-1}h(t)dt\\
   \nonumber&+2MT^{1-\alpha}\delta^{1-\alpha}\int_0^th(s)ds
   +MT^{1-\alpha}\int_t^{t+\delta}(t+\delta-s)^{\alpha-1}h(s) ds\\
   \nonumber\leq &\sup_{0\leq s\leq t\leq T-\delta}\big\|q_s(t+\delta)-q_s(t)\big\| T^{1-\alpha}H\bigg(\int_0^t(t+\delta-s)^{\frac{\alpha-1}{1-\alpha_3}}dt\bigg)^{1-\alpha_3}
   +2MT^{1-\alpha}\delta^{1-\alpha}HT^{1-\alpha_3}\\
   \nonumber&+MT^{1-\alpha}H\bigg(\int_t^{t+\delta}(t+\delta-s)^{\frac{\alpha-1}{1-\alpha_3}}ds\bigg)^{1-\alpha_3}\\
   \nonumber=&\frac{ T^{1-\alpha}H\bigg[(t+\delta)^{\frac{\alpha-\alpha_3}{1-\alpha_3}}-\delta^{\frac{\alpha-\alpha_3}{1-\alpha_3}}\bigg]^{1-\alpha_3}\sup_{0\leq s\leq t\leq
   T-\delta}\big\|q_s(t+\delta)-q_s(t)\big\|}{\big(\frac{\alpha-\alpha_3}{1-\alpha_3}\big)^{1-\alpha_3}}\\
   \nonumber&+2MT^{1-\alpha}\delta^{1-\alpha}HT^{1-\alpha_3}
   +\frac{MT^{1-\alpha}H\delta^\frac{\alpha-\alpha_3}{1-\alpha_3}}{\big(\frac{\alpha-\alpha_3}{1-\alpha_3}\big)^{1-\alpha_3}}\\
   \nonumber\leq&\frac{
   T^{1-\alpha}H\bigg[(t+\delta-\delta)^{\frac{\alpha-\alpha_3}{1-\alpha_3}}\bigg]^{1-\alpha_3}\sup_{0\leq
   s\leq t\leq
   T-\delta}\big\|q_s(t+\delta)-q_s(t)\big\|}{\big(\frac{\alpha-\alpha_3}{1-\alpha_3}\big)^{1-\alpha_3}}\\
   \nonumber&+2MT^{1-\alpha}\delta^{1-\alpha}HT^{1-\alpha_3}
   +\frac{MT^{1-\alpha}H\delta^{\alpha-\alpha_3}}{\big(\frac{\alpha-\alpha_3}{1-\alpha_3}\big)^{1-\alpha_3}}\\
   \leq&\frac{
   HT^{1-\alpha_3}\sup_{0\leq
   s\leq t\leq
   T-\delta}\big\|q_s(t+\delta)-q_s(t)\big\|}{\big(\frac{\alpha-\alpha_3}{1-\alpha_3}\big)^{1-\alpha_3}}
   +2MT^{1-\alpha}\delta^{1-\alpha}HT^{1-\alpha_3}
   +\frac{MT^{1-\alpha}H\delta^{\alpha-\alpha_3}}{\big(\frac{\alpha-\alpha_3}{1-\alpha_3}\big)^{1-\alpha_3}}.
\end{align}
Since the function $t\mapsto t^{1-\alpha}S(t)$ is uniformly
continuous over $[0,T]$, we have that
\begin{align*}
    \lim_{\delta\rightarrow 0^+}\sup_{0\leq s\leq t\leq T-\delta}\big\|q_s(t+\delta)-q_s(t)\big\|=0.
\end{align*}
Let $\delta\rightarrow 0^+$, the right side of inequality
(\ref{conti}) tends to zero. We obtain that the function
$$t^{1-\alpha}\int_0^tS(t-s)f\big(s,s^{1-\alpha}u(s),(Ku)(s)\big)ds$$
is continuous on $[0,T]$. Observe that the function $t\mapsto t^{1-\alpha}(x-g(u))$ is
continuous on $[0,T]$.
The combination of (\ref{11}) and (\ref{conti}) implies
that $t^{1-\alpha}\mathcal {N} u)(t)$ is continuous on $[0,T]$,
which implies that $\mathcal {N} u\in C_{1-\alpha}([0,T],X)$.

Next, we shall prove that the operator $\mathcal {N}:
C_{1-\alpha}([0,T],X)\rightarrow C_{1-\alpha}([0,T],X)$ is
contraction mapping on $C_{1-\alpha}([0,T],X)$. Let $u,\ v\in
C_{1-\alpha}([0,T],X).$ We compute
\begin{align*}
    &\|\mathcal {N} u-\mathcal {N} v\|_*\\
    \leq&\max_{t\in[0,T]}\|t^{1-\alpha}S(t)\|\|g(u)-g(v)\|\\
    &+
    \max_{t\in[0,T]}t^{1-\alpha}\int_0^t\|S(t-s)\|\big\|f\big(s,s^{1-\alpha}u(s),(Ku)(s)\big)
    -f\big(s,s^{1-\alpha}v(s),(Kv)(s)\big)\big\|ds\\
    \leq
    &Mb\|u-v\|_*+M\max_{t\in[0,T]}
    t^{1-\alpha}\int_0^t(t-s)^{\alpha-1}m_1(s)\|s^{1-\alpha}\big(u(s)-v(s)\big)\|ds\\
    &+M\max_{t\in[0,T]}
    t^{1-\alpha}\int_0^t(t-s)^{\alpha-1}m_2(s)\int_0^s\|k(s,\sigma)(u(\sigma)-v(\sigma))\|d\sigma ds\\
    \leq & Mb\|u-v\|_*+M\max_{t\in[0,T]}
    t^{1-\alpha}\int_0^t(t-s)^{\alpha-1}m_1(s)ds\|u-v\|_*\\
    &+MN\max_{t\in[0,T]}
    t^{1-\alpha}\int_0^t(t-s)^{\alpha-1}m_2(s)\int_0^s\|u(\sigma)-v(\sigma)\|d\sigma ds\\
    \leq & Mb\|u-v\|_*+MM_1\max_{t\in[0,T]}
    t^{1-\alpha}\bigg(\int_0^t(t-s)^{\frac{\alpha-1}{1-\alpha_1}}ds\bigg)^{1-\alpha_1}\|u-v\|_*\\
    &+\frac{MN}{\alpha}\max_{t\in[0,T]}
    t^{1-\alpha}\int_0^t(t-s)^{\alpha-1}m_2(s)s^\alpha ds\|u-v\|_*\\
    \leq & Mb\|u-v\|_*+\frac{MM_1\max_{t\in[0,T]}
    t^{1-\alpha_1}}{
    \big(\frac{\alpha-\alpha_1}{1-\alpha_1}\big)^{1-\alpha_1}}
    \|u-v\|_*\\
    &+\frac{MNT^\alpha}{\alpha}\max_{t\in[0,T]}
    t^{1-\alpha}\int_0^t(t-s)^{\alpha-1}m_2(s)ds\|u-v\|_*\\
    \leq&Mb\|u-v\|_*+\frac{MM_1\max_{t\in[0,T]}
    t^{1-\alpha_1}}{
    \big(\frac{\alpha-\alpha_1}{1-\alpha_1}\big)^{1-\alpha_1}}
    \|u-v\|_*\\
    &+\frac{MNM_2T^\alpha}{\alpha}\max_{t\in[0,T]}
    t^{1-\alpha}\bigg(\int_0^t(t-s)^{\frac{\alpha-1}{1-\alpha_2}}ds\bigg)^{1-\alpha_2}\|u-v\|_*\\
    \leq & Mb\|u-v\|_*+
    \frac{MM_1T^{1-\alpha_1}}{\big(\frac{\alpha-\alpha_1}{1-\alpha_1}\big)^{1-\alpha_1}}
    \|u-v\|_*+\frac{MM_2N T^{\alpha+1-\alpha_2}}{\alpha\big(\frac{\alpha-\alpha_2}{1-\alpha_2}\big)^{1-\alpha_2}}
    \|u-v\|_*.
\end{align*}
So we have
\begin{align*}
    \|\mathcal {N} u-\mathcal {N} v\|_*
    \leq \Omega \|u- v\|_*.
\end{align*}
By Banach contraction principle, we can obtain that $\mathcal {N}$ has an
unique fixed point which is just the solution of system (\ref{equ}).
\end{proof}

\begin{theorem}\label{3.5}
Assume that $A$ generates a fractional resolvent and $H_1$-$H_3$
hold. If $Mb<1$ and there exists an $r>0$
such that
\begin{align*}
  M[\|x\|+\max_{u\in B_r}\|g(u)\|]+
\frac{MT^{1-\alpha_3}}{\bigg(\frac{\alpha-\alpha_3}{1-\alpha_3}\bigg)^{1-\alpha_3}
} \|h\|_{L^{\frac{1}{\alpha_3}}}\leq r.
\end{align*}
Then system (\ref{equ}) has at least one solution.

\end{theorem}

\begin{proof}\ \
We consider the operator $\mathcal
{N}:C_{1-\alpha}([0,T],X)\rightarrow C_{1-\alpha}([0,T],X)$ defined
by (\ref{mainoperator}). From the proof of the above theorem, we
know $\mathcal {N}$ is well defined. We dived $\mathcal {N}$ into
two operators
\begin{align*}
   \left\{
     \begin{array}{ll}
       (\mathcal {A}u)(t):= S(t)[x-g(u)], & \hbox{ } \\
       (\mathcal {B}u)(t):=
\int_0^tS(t-s)f\big(s,s^{1-\alpha}u(s),(Ku)(s)\big)ds. & \hbox{}
     \end{array}
   \right.
\end{align*}

Let $u$, $v\in B_r$. Assume that $t\in (0,T]$. By the inequality
(\ref{zero}), we have
\begin{align*}
&\|t^{1-\alpha}(\mathcal {A}u)(t)-t^{1-\alpha}(\mathcal {B}v)(t)\|\\
 \leq &\|t^{1-\alpha}S(t)\|[\|x\|+\|g(u)\|]+\bigg\|t^{1-\alpha}\int_0^tS(t-s)f\big(s,s^{1-\alpha}v(s),(Kv)(s)\big)ds\bigg\|\\
\leq&M[\|x\|+\max_{u\in B_r}\|g(u)\|]+
\frac{MT^{1-\alpha_3}}{\bigg(\frac{\alpha-\alpha_3}{1-\alpha_3}\bigg)^{1-\alpha_3}
} \|h\|_{L^{\frac{1}{\alpha_3}}}\leq r,
\end{align*}
which implies that $(i)$ of Theorem \ref{the} holds.

For any $u,\ v\in B_r$, we have that
\begin{align*}
    \|t^{1-\alpha}(\mathcal {A}u)(t)-t^{1-\alpha}(\mathcal
{A}v)(t)\| = & \|t^{1-\alpha}S(t)g(u)-t^{1-\alpha}S(t)g(v) \| \\
\leq &Mb \|u-v\|_*.
\end{align*}
The assumption $Mb<1$ implies that $(ii)$ of
Theorem \ref{the} holds.

Assume that $u_n, u\in B_r$, $n=1,2,\cdots$, $u_n \rightarrow u$ in
the norm of $B_r$. From the proof of Theorem \ref{3.4}, we derive
that
\begin{align*}
    \|t^{1-\alpha}(\mathcal {B}u_n)-t^{1-\alpha}(\mathcal {B}u)\|
\leq
\frac{MM_1T^{1-\alpha_1}}{\big(\frac{\alpha-\alpha_1}{1-\alpha_1}\big)^{1-\alpha_1}}
    \|u_n-u\|_*+\frac{MM_2N T^{\alpha+1-\alpha_2}}{\alpha\big(\frac{\alpha-\alpha_2}{1-\alpha_2}\big)^{1-\alpha_2}}
    \|u_n-u\|_*.
\end{align*}
Then, $\mathcal {B}$ is continuous. The combination of inequality
(\ref{zero}) and inequality (\ref{conti}) implies that $\mathcal
{B}$ is uniformly bounded and equicontinuous. By the Arzela-Ascoli's
theorem, $\mathcal {B}$ is compact. This means that $(iii)$ of
Theorem \ref{the} holds. The proof is completed directly by Theorem
\ref{the}.
\end{proof}

\begin{example}
As an application, we consider the following partial differential equations with Dirichlet
boundary conditions.
\begin{align}\label{expa}
    \left\{
      \begin{array}{ll}
        D_t^\alpha u(t,x)=k^2\frac{\partial^2 }{\partial x^2}u(t,x)+\frac{\mu_1t^{1-\alpha}|u(t,x)|}{1+t^{1-\alpha}|u(t,x)|}
        +\frac{\mu_2}{1+\big|\int_0^te^{t-s}u(s,x)ds\big|}, \ t\in (0,1], x\in (0,1) \\
        u(t,0)=u(t,1)=0, & \hbox{ }\\
       \lim_{t\rightarrow 0^+} J^{1-\alpha}_t u(t,x)=p(x)+g(u(\cdot,x)). & \hbox{ }
      \end{array}
    \right.
\end{align}
\end{example}

In order to write the system (\ref{expa}) as the abstract form of
system $(\ref{equ})$, we take

$\bullet$ $X=L^2(0,\pi)$;

$\bullet$ $A=k^2\frac{\partial^2}{\partial x^2}$ with domain
$D(A)=\{g\in W^{2,2}(0,1): g(0)=g(1)=0\}$.

$\bullet$ $r(t,s)=e^{t-s},\ 0\leq s\leq t\leq 1.$

$\bullet$ $f:I\times X\times X\rightarrow X$ defined by
\begin{align*}
    f(t,w_1(\cdot),w_2(\cdot))=\frac{\mu_1|w_1(\cdot)|}{1+|w_1(\cdot)|}+\frac{\mu_2}{1+|w_2(\cdot)|},\ t>0, \ w_1(\cdot), w_2(\cdot)\in X.
\end{align*}

Observe that $A$ is closed, densely defined and has eigenvalues
$\{\lambda_n=-k^2n^2\pi^2\}_{n\in N}$ with eigenfunctions $\{sin(n
x)\}_{n\in N}$. Moreover, we can obtain $\rho(A)=\mathbb{C}/
\{sin(kn\pi x)\}_{n\in N} $. For $g(x)=\sum_{n=1}^{\infty} g_n sin
(kn\pi x)$, we define the family $\{S(t)\}_{t>0}$ by
\begin{align*}
(S(t)g)(x)=\sum_{n=1}^{\infty}t^{1-\alpha}E_{\alpha,\alpha}(-k^2n^2\pi^2
t^\alpha)g_{n}\sin (kn\pi x).
\end{align*}
By \cite[Example 3.1]{Li2012a}, it follows that $\{S(t)\}_{t>0}$ generates a
fractional resolvent. By \cite{Schneider1996}, it follows that the function
$E_{\alpha,\alpha}(-\cdot)$ is complete monotonicity thereby monotone nonincreasing function over $(0,+\infty)$. Since $t\mapsto E_{\alpha,\alpha}(-\cdot)$ is continuous on $[0,+\infty)$,
function $E_{\alpha,\alpha}(-\cdot)$ is nonincreasing on $[0,+\infty)$. We
compute
\begin{align*}
    \|t^{1-\alpha}S(t)\|^2=&\sup_{g\in L^2(0,1)}\|t^{1-\alpha}T(t)g\|^2\\
            =&\sup_{g\in L^2(0,1),\|g\|=1}\int_0^1|(T(t)g)(x)|^2dx\\
            =&\sup_{g\in L^2(0,1),\|g\|=1}\sum_{n=1}^{\infty}\big(E_{\alpha,\alpha}(-k^2n^2\pi^2
t^\alpha)\big)^2\big(g_{n}\big)^2\\
\leq & \big(E_{\alpha,\alpha}(-k^2\pi^2 t^\alpha)\big)^2 \sup_{g\in
L^2(0,1),\|g\|=1}\sum_{n=1}^{\infty}\big(g_{n}\big)^2\\
\leq &\frac{1}{\Gamma(\alpha)^2}.
\end{align*}
This means that $M\leq \frac{1}{\Gamma(\alpha)}.$\\
We compute
\begin{align*}
    &\|f(t,w_1(\cdot),w_2(\cdot))-f(t,v_1(\cdot),v_2(\cdot))\|\\
    \leq &\mu_1\bigg\|\frac{|w_1(\cdot)|}{1+|w_1(\cdot)|}-\frac{|v_1(\cdot)|}{1+|v_1(\cdot)|}\bigg\|
    +\mu_2\bigg\|\frac{1}{1+|w_2(\cdot)|}-\frac{1}{1+|v_2(\cdot)|}\bigg\|\\
    \leq& \mu_1\|w_1(\cdot)-v_1(\cdot)\|+\mu_2\|w_2(\cdot)-v_2(\cdot)\|.
\end{align*}\\
(I) In the case that $g(u(\cdot,x))=\Sigma^l_{i=1}a_iu(t_i,x)$, we have
\begin{align*}
\|\Sigma^l_{i=1}a_iu(t_i,\cdot)-\Sigma^l_{i=1}a_iv(t_i,\cdot)\| \leq
&\Sigma^l_{i=1}a_it_i^{\alpha-1}\big\|t_i^{1-\alpha}u(t_i,\cdot)-t_i^{1-\alpha}v(t_i,\cdot)\big\|\\
\leq & \Sigma^l_{i=1}a_it_i^{\alpha-1}\|u-v\|_*;
\end{align*}
by Theorem \ref{3.4}, if there exist $\alpha_1,\ \alpha_2\in (0,\alpha)$ such that
\begin{align*}
    \frac{\mu_1}{\Gamma(\alpha)\bigg(\frac{\alpha-\alpha_1}{1-\alpha_1}\bigg)^{1-\alpha_1}}
    +\frac{\Sigma_{i=1}^la_it_i^{\alpha-1}}{\Gamma(\alpha)}+\frac{\mu_2e}{\Gamma(\alpha+1)
    \bigg(\frac{\alpha-\alpha_2}{1-\alpha_2}\bigg)^{1-\alpha_2}}<1,
\end{align*}
then system (\ref{equ}) has a unique solution.\\
(II) In the case that $$g(u(\cdot,x))=\frac{\Sigma^l_{i=1}a_i|u(t_i,x)|}{1+\Sigma^l_{i=1}a_i|u(t_i,x)|},$$ we have
\begin{align*}
\|g(u(\cdot,\cdot))-g(v(\cdot,\cdot))\| \leq
&\bigg\|\frac{\Sigma^l_{i=1}a_i|u(t_i,\cdot)|}{1+\Sigma^l_{i=1}a_i|u(t_i,\cdot)|}
-\frac{\Sigma^l_{i=1}a_i|v(t_i,\cdot)|}{1+\Sigma^l_{i=1}a_i|v(t_i,\cdot)|}\bigg\|\\
\leq&\big\|\Sigma^l_{i=1}a_i|u(t_i,\cdot)|-\Sigma^l_{i=1}a_i|v(t_i,\cdot)|\big\|\\
\leq&\Sigma^l_{i=1}a_it_i^{\alpha-1}\|t_i^{1-\alpha}u(t_i,\cdot)-t_i^{1-\alpha}v(t_i,\cdot)\|\\
\leq&\Sigma^l_{i=1}a_it_i^{\alpha-1}\|t_i^{1-\alpha}\|u-v\|_*;
\end{align*}
by Theorem \ref{3.5}, if
\begin{align*}
    \frac{\mu_1}{\Gamma(\alpha)\bigg(\frac{\alpha-\alpha_1}{1-\alpha_1}\bigg)^{1-\alpha_1}}<1,
\end{align*}
then system (\ref{equ}) has at least one solution.

\end{document}